\documentclass[12pt,oneside]{amsart}
\usepackage{mathrsfs,amssymb, array, longtable}
\usepackage[matrix,arrow]{xy}

\makeatletter \@addtoreset{equation}{section} \makeatother

\newcommand{\muu}{{\boldsymbol{\mu}}}

\newcommand{\Sing}{\operatorname{Sing}}
\newcommand{\an}{\operatorname{an}}
\newcommand{\Supp}{\operatorname{Supp}}
\newcommand{\wt}{\operatorname{wt}}
\newcommand{\mt}[1]{\operatorname{#1}}
\newcommand{\ov}[1]{\overline{#1}}
\newcommand{\B}{{\mathbf B}}
\newcommand{\CC}{\mathbb{C}}
\newcommand{\QQ}{\mathbb{Q}}
\newcommand{\ZZ}{\mathbb{Z}}
\newcommand{\PP}{\mathbb{P}}
\newcommand{\OOO}{{\mathscr{O}}} 
\newcommand{\down}[1]{\left\lfloor #1\right\rfloor}
\newcommand{\basket}[1]{(#1)}
\newcommand{\type}[1]{$I_{#1}$}

\def\stackunder#1#2{\mathrel{\mathop{#2}\limits_{#1}}}%

\renewcommand\labelenumi{(\roman{enumi})}
\newcommand{\comment}[1]{}

\newcommand{\xref}[1]{{\rm \ref{#1}}}

\newtheorem{theorem}[equation]{Theorem}
\newtheorem{proposition}[equation]{Proposition}
\newtheorem{lemma}[equation]{Lemma}
\newtheorem{corollary}[equation]{Corollary}
\newtheorem{conjecture}[equation]{Conjecture}

\theoremstyle{definition}

\newtheorem{example}[equation]{Example}
\newtheorem{discussion}[equation]{}
\newtheorem{construction}[equation]{Construction}
\newtheorem{computation}[equation]{}

\newtheorem{notation}[equation]{Notation}
\newtheorem*{warning}{Warning}

\title{Multiple fibers of del Pezzo fibrations}
\author{Shigefumi Mori}
\author{Yuri Prokhorov}

\address{Shigefumi Mori: RIMS, 
Kyoto University, Oiwake-cho, Kitashirakawa, Sakyo-ku, Kyoto
606-8502, Japan}
\email{mori@kurims.kyoto-u.ac.jp}
\address{Yuri Prokhorov: Department 
of Algebra, Faculty of Mathematics, Moscow State
University, Moscow 117234, Russia}
\email{prokhoro@mech.math.msu.su}

\thanks {
The research of the first author was supported by JSPS Grant-in-Aid for
Scientific Research (B)(2), Nos. 16340004 and 20340005.
The second author was
partially supported by 
RFBR, Nos. \ 08-01-00395-a and 06-01-72017-MHTI-a.}

\begin{document}
\maketitle
\begin{abstract}
We prove that a terminal three-dimensional del Pezzo fibration 
has no fibers of multiplicity 
$\ge 6$.
We also obtain a rough classification possible configurations of singular points 
on multiple fibers and give some examples.
\end{abstract}

\section{Introduction}
Throughout this paper a \textit{weak del Pezzo fibration} is
a projective morphism $f\colon X\to Z$ with connected fibers 
from a threefold $X$ with 
terminal singularities to a smooth curve $Z$ 
such that $-K_X$ is $f$-nef and $f$-big near a general fiber. 
If additionally $-K_X$ is $f$-ample, we say that $f\colon X\to Z$
is a \textit{del Pezzo bundle}.
(We do not assume that $X$ is $\QQ$-factorial nor $\rho(X/Z)=1$).
The main reason to study del Pezzo fibrations
comes from the three-dimensional birational geometry,
namely the class of del Pezzo bundles with $\QQ$-factorial singularities and 
relative Picard number one is one of three possible outcomes of 
the minimal model program for threefolds of negative Kodaira dimension.

Our main result is the following.
\begin{theorem}
\label{main}
Let $f\colon X\to Z$ be a weak del Pezzo fibration and let $f^*(o)=m_oF_o$
be a special fiber of multiplicity $m_o$.
Then $m_o\le 6$. 
Moreover, all the cases $1\le m_o\le 6$ occur.
Furthermore, 
let $\B(F_o)=\basket{r_1,\dots,r_n}$ be the basket of singular 
points of $X$ at which $F_o$ is not Cartier.
Then, in the case $m_o\ge 2$, there are only the following 
possibilities:
\begin{enumerate}
\item 
$m_o= 2$, $\B(F_o)=\basket{8}$, 
$\basket{2,6}$, 
$\basket{4,4}$,
$\basket{2, 2, 4}$, or
$\basket{2, 2, 2, 2}$,
\item 
$m_o= 3$, $\B(F_o)=\basket{9}$,
$\basket{3, 3, 3}$, or
$\basket{3,6}$, 
\item 
$m_o= 4$, $\B(F_o)=\basket{2, 4, 4}$, 
\item 
$m_o= 5$, $\B(F_o)=\basket{5,5}$,
\item 
$m_o= 6$, $\B(F_o)=\basket{2, 3, 6}$.
\end{enumerate}
The possible types of 
singularities in $\B(F_o)$, 
the local behavior of $F_o$ near singular points, and 
the possible types of 
a general fiber
are collected in Table \xref{table}.
\end{theorem}

\begin{warning}
In the statement of Theorem \ref{main}
and Table \xref{table} we do not assert that
the basket $\B(F_o)$ contains all the singularities along 
$F_o$. It is possible that $F_o$ is Cartier at some non-Gorenstein points
(see Example \ref{ex-not-basket}).
\end{warning}

\begin{table}[h]
\renewcommand{\extrarowheight}{5pt}
\renewcommand{\arraystretch}{1.3}
\caption{}
\label{table}
\begin{tabular}{|l|l|l|l|l|c|}
\hline
type&$m_o$&$\B(F_o)=$&$(b_1,\dots,b_n)$&$q_i$&$K_{F_g}^2$
\\[-10pt]
&&$(r_1,\dots,r_n)$&&&
\\
\hline\hline
\type{2, 3, 6}&
$6$& 
$\basket{ 2, 3, 6}$
&
$(1,\pm 1,\pm 1)$
&
$q_i\equiv -1$
&
6
\\ \hline
\type{5, 5}&
$5$
& 
$\basket{5, 5}$
&$b_1^2+b_2^2\equiv 0$
&
$q_i\equiv -1$
&
5
\\ \hline
\type{2, 4, 4}&
$4$
& 
$\basket{2, 4, 4}$
&
$(1,\pm 1,\pm 1)$
&
$q_i\equiv -1$
&
4,\, 8
\\ \hline
\type{3, 3, 3}&
$3$
& 
$\basket{3, 3, 3}$
&
$(\pm 1,\pm 1,\pm 1)$
&
$q_i\equiv -1$
&
3,\, 6,\, 9
\\ \hline
\type{2, 2, 2, 2}&
$2$
&
$\basket{2, 2, 2, 2}$
&
$(1,1,1,1)$ 
&
$q_i\equiv 1$ 
&
even
\\ \hline
\type{3, 6}&
$3$
& 
$\basket{3, 6}$
&
$(\pm 1,\pm 1)$
&
$q_i\equiv 4$
&
3,\, 6,\, 9
\\ \hline
\type{9}&
$3$ 
&
$\basket{9}$ 
&
$b_1=\pm 2q_1/3$
& 
$q_1=3$ or $6$
&
$\equiv q_1/3\mod 3$ 
\\ \hline
\type{2, 2, 4}&
$2$
&
$\basket{2, 2, 4}$
&
$(1,1,\pm 1)$
&
$q_i\equiv r_i/2$
&
odd
\\ \hline
\type{4, 4}&
$2$&$\basket{4, 4}$
&
$(\pm 1,\pm 1)$
&
$q_i\equiv r_i/2$
&
even
\\ \hline
\type{2, 6}&
$2$& $\basket{2, 6}$
&
$(1, \pm 1)$
&
$q_i\equiv r_i/2$
&
even
\\ \hline
\type{8}&
$2$&$\basket{8}$
&
$(\pm 1)$ or $(\pm 3)$ 
&
$q_i\equiv r_i/2$
&
odd
\\ \hline
\end{tabular}
\end{table}

The idea of the proof is easy. In fact, 
it is sufficient to compute dimensions of 
linear systems $\dim |lF_o|$ 
by using the orbifold Riemann-Roch formula \cite{Reid-YPG1987}.
The main theorem is proved in \S\S \ref{sect-prep} -- \ref{sect-proof}.
In \S \ref{sect-exam} we give some examples.
In fact, it will be shown that 
all cases in Table \ref{table}
except possibly for cases \type{2,6} and \type{8} occur. 
Finally, in \S \ref{sect-45} we discuss 
fibers of multiplicity $5$ and $6$.

\par\medskip\noindent
\textbf{Notation in Table \ref{table}.}
The number $b_k$ in the fourth column 
is the weight which appears in a singularity 
$\frac{1}{r_k}(1,-1, b_k)\in \B(F_o)$,
$q_k$ in the fifth column is an integer such that $F_o\sim q_kK_X$ near $P_k\in \B(F_o)$.
$F_g$ in the final column denotes a general fiber of $f$.
We say that the fiber $f^*(o)=m_oF_o$
is of type \type{r_1,\dots,r_n} if $\B(F_o)=\basket{r_1,\dots,r_n}$.

We also say that the fiber $F_o$
is \textit{regular} if it is of type \type{2,3,6},
\type{5,5}, \type{2,4,4}, \type{3,3,3}, or
\type{2,2,2,2}. Otherwise, $F_o$ is said to be \textit{irregular}.

\par\medskip\noindent
\textbf{Acknowledgments.}
Much of this 
work was done at the Research Institute for Mathematical
Sciences (RIMS), Kyoto University in February 2008. 
The second author would like to
thank RIMS for invitation to work there, for
hospitality and stimulating working environment.

\section{Preliminaries}
\label{sect-prelim}
\begin{discussion} 
\textbf{Terminal singularities \cite{Mori-1985-cla}, \cite{Reid-YPG1987}.}
Let $(X,P)$ be a three-dimensional terminal singularity
of index $r$ and let $D$ be a Weil $\QQ$-Cartier divisor on $X$.
\end{discussion}

\begin{lemma}[{\cite[Corollary 5.2]{Kawamata-1988-crep}}]
\label{lemma-Kawamata-index-divisors}
In the above notation, 
there is an integer $i$ such that $D\sim iK_X$ near $P$.
In particular, $rD$ is Cartier.
\end{lemma}

\begin{discussion} 
\label{not-01}
Notation as above. 
There is a deformation $X_\lambda$ of $X$ such that
$X_\lambda$ has only cyclic quotient singularities
$(X_\lambda,P_{\lambda,k})\simeq \frac {1}{r_k}(1,-1,b_k)$,
$0<b_k<r_k$, $\gcd(b_k, r_k)=1$.
Thus, to every theefold $X$ with terminal singularities,
one can associate 
a collection $\B=\basket{(r_{P,k},b_{P,k})}$, where
$P_{\lambda,k}\in X_{\lambda,k}$ is a singularity
of type $\frac 1{r_{P,k}}(1,-1,b_{P,k})$.
This collection is called 
the \textit{basket} of singularities of $X$. By abuse of 
notation, we also will write $\B=\basket{r_{P,k}}$ instead of 
$\B=\basket{(r_{P,k},b_{P,k})}$.
The index of $(X,P)$ is the least common multiple
of indices of points $P_{\lambda,k}$.
For any Weil divisor $D$, $\B(D)\subset \B$ denotes the collection of 
points where $D$ is not Cartier.

Deforming $D$ with $(X,P)$ 
we obtain Weil divisors $D_{\lambda}$ on $X_\lambda$.
Thus we have a collection of numbers $q_k$ such that
$0\le q_k<r_k$ and 
$D_{\lambda}\sim q_kK_{X_\lambda}$ near $P_{\lambda,k}$.

\end{discussion}

\begin{discussion} 
\textbf{Orbifold Riemann-Roch formula \cite{Reid-YPG1987}.}
Let $X$ be a threefold with terminal singularities 
and let $D$ be a Weil $\QQ$-Cartier divisor on $X$.
Then
\begin{multline}
\label{eq-RR}
\chi(D)=
\frac1{12}D\cdot (D-K_X)\cdot (2D-K_X)+
\\
+\frac1{12}D\cdot c_2(X)+\chi(\OOO_X)+\sum_{P\in \B} c_P(D),
\end{multline}
where 
\begin{equation}
\label{eq-RR-cP-def}
c_P(D)=-q_P\frac{r_P^2-1}{12r_P}+\sum_{j=1}^{q_P-1}\frac{\ov{b_Pj}(r_P-\ov{b_Pj})}{2r_P},
\end{equation}
$q_P$ is such as in \ref{not-01}, and 
$\overline{\phantom{X}}$ denotes the smallest residue $\mod r_P$.

Assume that 
$D^2\equiv 0$. Then 
\begin{equation}
\label{eq-RR-D2=0} 
\chi(D)=
\frac1{12}D\cdot K_X^2+\frac1{12} D\cdot c_2(X)+\chi(\OOO_X)
+\sum c_P(D).
\end{equation}
We have (see, e.~g., \cite[proof of 2.13]{Alexeev-1994ge})
\begin{equation}
\label{eq-RR-cPK-K}
c_P(-K)=\frac{r_P^2-1}{12r_P}-\frac{b_P(r_P-b_P)}{2r_P},
\qquad 
c_P(K)=-\frac{r_P^2-1}{12r_P}.
\end{equation}
\end{discussion}

\begin{construction}[Base change]
\label{base-change}
Let $f\colon X\to Z$ be a weak del Pezzo fibration
and let $f^*(o)=m_oF_o$
be a special fiber of multiplicity $m_o$.
Regard $f\colon X\to (Z,o)$ as a germ. Let 
$(\CC,0)\simeq (Z',o')\to (Z,o)\simeq (\CC,0)$ is given by $t \mapsto t^{m_o}$
and let $X'$ be the normalization of $X\times_Z Z'$. We obtain the 
following commutative diagram:
\begin{equation}
\label{eq-base-change}
\xymatrix{
X'\ar[r]^{\pi}\ar[d]_{f'}&X\ar[d]^{f}
\\
Z'\ar[r]&Z
}
\end{equation}
Here $f'$ is a weak del Pezzo fibration with 
special fiber $F_o'=f'^{*}o'=\pi^{*}F_o$ of multiplicity $1$ and
$\pi$ is a $\muu_{m_o}$-cover which is 
\'etale outside of the set $M$ of points where 
$F_o$ is not Cartier.
Hence there is a $\muu_{m_o}$-action on 
$X'$ such that $X=X'/\muu_{m_o}$ and the action is free 
outside of $M$.

Conversely, let $f'\colon X'\to Z'\ni o'$ be a weak del Pezzo fibration 
with central fiber of multiplicity $1$. Assume that $f$
equipped with an equivariant $\muu_{m_o}$-action such that 
the action on $X'$ is \'etale in codimension two. 
If the quotient $X'/\muu_{m_o}$ has only terminal singularities, 
then $X'/\muu_{m_o}\to Z'/\muu_{m_o}$ is a weak del Pezzo fibration
with special fiber of multiplicity $m_o$.
\end{construction}

\begin{proposition}
\label{cor-point}
Let $f\colon X\to Z$ be a weak del Pezzo fibration. Let $f^*(o)=m_oF_o$
be a special fiber of multiplicity $m_o$.
There is a point $P\in F_o$ such that the 
index of $F_o$ at $P$ is divisible by $m_o$. 
\end{proposition}
\begin{proof}
Regard $f\colon X\to (Z,o)$ as a germ and apply 
Construction \ref{base-change}.
It is sufficient to show that 
$\muu_{m_o}$ has a fixed point on $F_o'$ (see Lemma \ref{lemma-Kawamata-index-divisors}).

First we consider the case of del Pezzo bundle, i.e., the case where $-K_X$ is ample.
Let $\gamma$ be the log canonical threshold of $(X',F_o')$
and let $W'\subset X'$ be a minimal center of log canonical 
singularities of $(X',\gamma F_o')$ (see \cite[\S 1]{Kawamata1997}).

Assume that $\dim W'\le 1$. 
Let $H$ be a general hyperplane section of $X$
passing through 
$\pi(W')$ and let $H':=\pi^*H$.
For $0<\varepsilon\ll 1$, the pair $(X',\, \gamma F_o'+\varepsilon H')$
is not LC along $\pi^{-1}\pi(W')$ and LC outside.
Therefore for some $0<\delta\ll \epsilon$ 
the pair $(X',\, (\gamma-\delta)F_o'+\varepsilon H')$
is not KLT along $\pi^{-1}\pi (W')$ and KLT outside.
Moreover, $W'$ is a minimal LC center for $(X',\, (\gamma-\delta)F_o'+\varepsilon H')$.
Recall that
any irreducible component of 
the intersection of two LC centers  
is also an LC center  \cite[Proposition 1.5]{Kawamata1997}.
Hence $W'$ is the only LC center in its 
neighborhood. Since the boundary 
$(\gamma-\delta)F_o'+\varepsilon H'$ is $\muu_{m_o}$-invariant,
all the $gW'$ for $g\in \muu_{m_o}$ are also
centers of log canonical singularities for 
the pair $(X',\, (\gamma-\delta)F_o'+\varepsilon H')$.
On the other hand, the locus $\muu_{m_o}W'$ 
of log canonical singularities
for the pair
$(X',\, (\gamma-\delta)F_o'+\varepsilon H')$ is connected,
see \cite[\S 5]{Shokurov-1992-e-ba}, \cite[17.4]{Utah}.
Hence $\muu_{m_o}W'$ is irreducible and so $W'$ is 
$\muu_{m_o}$-invariant. If $W'$ is a point, we are done.
Otherwise $W'$ is a smooth rational curve
\cite[Th. 1.6]{Kawamata1997},
\cite{Kawamata-1997-Adj}. 
But any cyclic group acting on $\PP^1$ has a fixed points.

Assume that $\dim W'=2$, 
that is, $W'=\down{\gamma F_o'}$ and the pair $(X',\gamma F_o')$ is PLT.
By the inversion of adjunction \cite[3.3]{Shokurov-1992-e-ba}, \cite[17.6]{Utah}
and Connectedness Lemma \cite[\S 5]{Shokurov-1992-e-ba}, \cite[17.4]{Utah}
the surface
$W'$ is irreducible, normal and has only KLT singularities.
Hence $W'$ is a KLT log del Pezzo surface.
In particular, $W'$ is rational.
Then the assertion follows by Lemma \ref{lemma-claim-fp}
below.

Now we consider the general case.
We apply $\muu_{m_o}$-equivariant MMP in the category 
$\muu_{m_o}$-threefolds (i.e., threefolds with terminal singularities 
and such that every $\muu_{m_o}$-invariant Weil divisor is $\QQ$-Cartier,
see e.g. \cite[0.3.14]{Mori-1988}). 
Let $X_1\to X'$ be a $\muu_{m_o}$-equivariant $\QQ$-factorialization.
Run $\muu_{m_o}$-equivariant MMP over $Z'$:
\[
X_1 \dashrightarrow X_2\dashrightarrow \cdots \dashrightarrow
X_N. 
\]
These maps induce a sequence maps 
\[
X_1/\muu_{m_o} \dashrightarrow X_2/\muu_{m_o}\dashrightarrow \cdots \dashrightarrow
X_N/\muu_{m_o}. 
\]
where each step is either $K$-negative divisorial contraction or a flip
(both are not neseccarily extremal).
Hence, on each step the quotient $X_i/\muu_{m_o}$ has only terminal singularities 
and the action of $\muu_{m_o}$ on $X_i$ is free in codimension two.
On the last step $X_N$ is either a del Pezzo bundle over $Z'$
with $\rho^{\muu_{m_o}}(X_N/Z')=1$ or a $\QQ$-conic bundle over a surface $S$ and
$S/Z'$ is a rational curve fibration.
In both cases $\muu_{m_o}$ has a fixed point on $X_N$.
We prove the existence of fixed point on $X_i$ by a descending induction on $i$.
So we assume that $X_{i+1}$ has a fixed point, say $P$.
If $\psi_i\colon X_{i} \dashrightarrow X_{i+1}$ is a flip, we may assume that
$P$ is contained in the flipped curve $C_{i+1}\subset X_{i+1}$. 
In this case $\muu_{m_o}$
acts on 
a connected closed subset of
the flipping curve $C_{i}\subset X_{i}$.
Since $C_{i}$ is a tree of rational curves, $\muu_{m_o}$
has a fixed point on $C_{i}$.
Similar argument works in the case where $\psi_i\colon X_{i} \to X_{i+1}$
is a contraction of a 
$\muu_{m_o}$-invariant
divisor to a curve.
Thus we may assume that $\psi_i\colon X_{i} \to X_{i+1}$
is a divisorial contraction that contracts a
$\muu_{m_o}$-invariant
divisor $E\subset X_i$
to $P$.
Let $\gamma$ be the log canonical threshold of $(X_i,E)$
and let $W_i\subset X_i$ be a minimal center of log canonical
singularities of $(X_i,\gamma E)$.
As in the del Pezzo bundle case above,
considering the action of $\muu_{m_o}$ on $W_i$ we find
a fixed point. This proves our proposition.
\end{proof}

\begin{lemma}
\label{lemma-claim-fp}
Let $S$ be a rational surface.
Then any action of a finite 
cyclic group on $S$ has a fixed point.
\end{lemma}

\begin{proof}
Let $\muu_m$ be the cyclic group acting on $S$.
Replacing $S$ with its normalization and the minimal 
resolution, we may assume that $S$ is smooth.
Since $S$ is rational, $H^i(S,\CC)=0$ if $i$ is odd.
Then the assertion follows by the Lefschetz fixed point formula.
\end{proof}

\section{Preparations}
\label{sect-prep}
\begin{notation}
\label{not}
Let $f\colon X\to Z$ be a weak del Pezzo fibration.
Compactify $X$ and $Z$ and resolve $X$ only above the added points of $Z$.
Thus we may assume that both $X$ and $Z$ are projective.
Let $F_g$ be a general fiber and let $f^*(o)=m_oF_o$
be a special fiber of multiplicity $m_o$. 
Write $m_o=m\alpha$, where $m$ and $\alpha$ are positive integers
and put $D:=\alpha F_o$. Then $m_oF_o=mD=f^*(o)$.
\end{notation}

\begin{computation}
By a variant of J. Koll\'ar's Higher Direct Images Theorem
(see \cite[1-2-7]{KMM}, \cite{Nakayama1986}), one has that
$R^i f_*\OOO_X(K_X-jD)$ is torsion free for all $i$.
But its restriction to the general fiber $F_g$ is zero for $i \neq2$
because $-K_{F_g}$ is nef and big.
Hence $R^i f_*\OOO(K_X-jD)=0$ for $i\neq 2$.
Further, the 
Leray spectral sequence yields 
\[
H^q(X,K_X-jD)=H^{q-2}(Z,R^2 f_*\OOO(K_X-jD))=0
\]
for $q-2 \neq1$
and $j \gg0$ because $R^2 f_* \OOO(K_X-jD)$ is very negative.
By Serre duality 
\[
H^{3-q}(X,jD)\simeq H^q(X,K_X-jD)^{\vee} =0
\]
for $q \neq 3$ and $j \gg0$.

Finally, $H^{i}(X,jD)=0$ for all $i>0$, $j> j_0\gg 0$. We also have 
\[
H^0(X,j f^*(o)+lD)\simeq H^0(X,j f^*(o))
\]
for $l=0,\dots,m-1$.
Put $j_1:=\down{j_0/m}$ and
\[
\Theta_l:=\frac{1}{mj_1 }h^0(X,
j_1 f^*(o))-\frac{1}{mj_1 +l} h^0(X,j_1 f^*(o)+lD).
\]
Thus for $l=0,\dots,m-1$ we have 
\begin{equation}
\label{eq-RR-e1}
\Theta_l=\frac{l}{mj_1 (mj_1 +l)}h^0(X,j_1 f^*(o))=\frac{l(j_1 -p_a+1)}{mj_1 (mj_1 +l)},
\end{equation}
where $p_a$ is the genus of $Z$. On the other hand, by 
\eqref{eq-RR-D2=0} 
\begin{equation}
\label{eq-RR-e2}
\Theta_l=-\frac{1}{mj_1 +l}\sum_{P\in \B} c_P(lD)+\frac{l}{mj_1 (mj_1 +l)}\chi(\OOO_X).
\end{equation}
Comparing \eqref{eq-RR-e1} and \eqref{eq-RR-e2} we get
\begin{equation}
\label{eq-RR-e3}
m\sum_{P\in \B} c_P(lD)=-l,\qquad l=0,\dots,m-1.
\end{equation}
\end{computation}

\begin{computation}
Denote
\[
\begin{array}{lll}
\Delta_a&:=&\chi(\OOO_X(-K-aF_o))-\chi(\OOO_X(-K-(a+1)F_o)),
\\[9pt]
\delta_a&:=&\sum\limits_{P\in \B} c_P(-K-aF_o)- \sum\limits_{P\in \B} c_P(-K-(a+1)F_o).
\end{array}
\]
As above, for $a=0,\dots,m_o-2$, the following equality holds
\begin{multline*}
\Delta_a=
\frac{13}{12}K^2\cdot F_o
+\frac{1}{12}F_o\cdot c_2(X)+
\sum_{P\in \B} c_P(-K-aF_o)-
\\
-\sum_{P\in \B} c_P(-K-(a+1)F_o)
=
\frac{13}{12m_o}K^2\cdot F_g 
+\frac{1}{12m_o}F_g\cdot c_2(X)+\delta_a. 
\end{multline*}
Since $K^2\cdot F_g=K_{F_g}^2$ and 
$F_g\cdot c_2(X)=c_2(F_g)=12-K_{F_g}^2$, we have
\begin{equation}
\label{eq-Delta}
\Delta_a=\frac{K_{F_g}^2+1}{m_o}+\delta_a.
\end{equation}
\end{computation}

\begin{discussion}
\textbf{Some computations.}
Let $(X,P)$ be a cyclic quotient singularity of type
$\frac 1r (a,-a,1)$, let $D$ be a Weil divisor on $X$,
and let $m$ be a natural number. We have $D\sim q K_X$ for some $0\le q < r$. Denote 
\begin{equation}
\label{eq-RR-su-def}
\Xi_{P,m}:=\sum_{l=1}^{m-1} c_P(lD).
\end{equation}
We also will write $\Xi_P$ or $\Xi$ instead of $\Xi_{P,m}$ if no confusion is likely.
By definition
\begin{equation}
\label{eq-S1}
\Xi_{P,m}=\sum_{l=1}^{m-1}
\left(-
\ov{ql}\frac{r^2-1}{12r}+ \sum_{j=1}^{\ov{ql}-1}\frac{\ov{bj}(r-\ov{bj})}{2r}\right).
\end{equation}
We compute $\Xi$ in some special situation:
\end{discussion}

\begin{lemma}
\label{lemma-Sp}
Let $s:=\gcd(r,q)$.
Write
$r=sm$ and $q=sk$ for some $s,\, k\in \ZZ_{>0}$
\textup(so that $\gcd(m,k)=1$\textup).
Then 
\begin{equation}
\label{eq-sum-f}
\Xi_{P,m}=-\frac{m^2-1}{24m}r.
\end{equation}
\end{lemma}
\begin{proof}
By our assumption $\gcd(m,k)=1$ the parameter $\ov{ql}=\ov{skl}$
runs through all the values $sl$, $l=1,\dots,m-1$. Hence,
\begin{equation*}
\Xi=-\sum_{l=1}^{m-1}
sl\frac{r^2-1}{12r}+\sum_{l=1}^{m-1} \sum_{j=1}^{sl-1}\frac{\ov{bj}(r-\ov{bj})}{2r}.
\end{equation*}
Since $\ov{bj}(r-\ov{bj})=\ov{bj'}(r-\ov{bj'})$ for $j+j'=r$, we have 
\[
\sum_{j=1}^{sl-1}\frac{\ov{bj}(r-\ov{bj})}{2r}
=
\sum_{j=r-sl+1}^{r-1}\frac{\ov{bj}(r-\ov{bj})}{2r}.
\]
Therefore,
\begin{multline*}
\Xi=-
\frac{m(m-1)s}{2}
\frac{r^2-1}{12r}+
\frac{m-1}{2}
\sum_{j=1}^{r-1}\frac{\ov{bj}(r-\ov{bj})}{2r}-
\frac12 \sum_{l=1}^{m-1}\frac{\ov{bsl}(r-\ov{bsl})}{2r}=
\\
-\frac{(m-1)r}{2}
\frac{r^2-1}{12r}+
\frac{m-1}{2}
\sum_{j=1}^{r-1}
\frac{\ov{bj}(r-\ov{bj})}{2r}
-\frac12 \sum_{l=1}^{m-1}\frac{sl(r-sl)}{2r}
=\frac{m-1}{2}c_P(rK)\\
-\frac{s}{4} \sum_{l=1}^{m-1}l
+\frac{s^2}{4r} \sum_{l=1}^{m-1} l^2=
-\frac{sm(m-1)}{8}+\frac{s^2}{24r}(m-1)m(2m-1)=
\\
=
-\frac{m-1}{8}
\left(r- \frac{s}{3}(2m-1)\right)=
-\frac{m-1}{24}
\left(r+\frac{r}{m}\right).
\end{multline*}
(We used $sm=r$ and $c_P(rK)=0$.) This proves our lemma.
\end{proof}

\begin{lemma}
\label{lemma-su-div}
If $m=m_1m_2$, where $m_1D$ is Cartier, then
\[
\Xi_{P,m}=m_2\Xi_{P,m_1}.
\]
\end{lemma}
\begin{proof}
Follows by \eqref{eq-RR-su-def} because $c_P(tD)$ 
is $r$-periodic. 
\end{proof}

\section{Proof of Theorem \ref{main}}
\label{sect-proof}
Notation as in \ref{not}.
Near each singular point $P\in X$ of index $r_P$ 
we write 
\[
D\sim q_PK_X.
\]
Then $mq_PK_X\sim mD$ is Cartier near $P$. Hence,
\begin{equation}
\label{eq-RR-div1}
mq_P \equiv 0 \mod r_P.
\end{equation}
From \eqref{eq-RR-e3} we have 
\begin{equation}
\label{eq-RR-su}
\sum_{P\in \B} \Xi_{P,m}=-\sum_{l=1}^{m-1} \frac lm
=-\frac{m-1}{2}.
\end{equation}

\begin{proposition}
\label{prop-Dprime}
Notation as above. If $m$ is prime, then
we have one of the following possibilities:
\begin{enumerate}
\renewcommand
\labelenumi{(\arabic{section}.\arabic{equation}.\arabic{enumi})}
\renewcommand
\theenumi{(\arabic{section}.\arabic{equation}.\arabic{enumi})}
\item 
\label{main-m2-b8}
\quad $m= 2$, $\B(D)=\basket{8}$, 
\item 
\label{main-m2-b26}
\quad $m= 2$, $\B(D)=\basket{2,6}$, 
\item 
\label{main-m2-b44}
\quad $m= 2$, $\B(D)=\basket{4,4}$,
\item 
\label{main-m2-b224}
\quad $m= 2$, $\B(D)=\basket{2, 2, 4}$,
\item 
\label{main-m2-b2222}
\quad $m= 2$, $\B(D)=\basket{2, 2, 2, 2}$,
\item 
\label{main-m3-b9}
\quad $m= 3$, $\B(D)=\basket{9}$,
\item 
\label{main-m3-b333}
\quad $m= 3$, $\B(D)=\basket{3, 3, 3}$,
\item 
\label{main-m3-b36}
\quad $m= 3$, $\B(D)=\basket{3,6}$, 

\item 
\label{main-m5-b55}
\quad $m= 5$, $\B(D)=\basket{5,5}$,

\item
\label{m5-b10}
\quad $m=5$, $\B(D)=\basket{10}$, 
\item
\label{m11-b11}
\quad $m=11$, $\B(D)=\basket{11}$.

\end{enumerate}. 
\end{proposition}

\begin{proof}
By \eqref{eq-RR-div1} we have $mq_P\equiv 0 \mod r_P$ and 
$r_P\equiv 0\mod m$ for all $P\in \B(D)$
(otherwise $q_P\equiv 0 \mod r_P$ and $P\notin \B(D)$).
Put $s_P:=r_P/m$. Then $q_P=s_Pk_P$ for some $k_P\in \ZZ_{>0}$.
Since $\gcd(k_P,q_P)=1$,
the assumption 
of Lemma \xref{lemma-Sp} holds for each point $P\in \B(D)$.
Combining \eqref{eq-sum-f} with \eqref{eq-RR-su} we obtain
\[
(m+1)\sum_{P\in \B} r_P=12m.
\]
Hence, $m\in \{2,\, 3,\, 5,\, 11\}$.
Using the fact 
$r_P\equiv 0\mod m$ we get the statement.
\end{proof}

\begin{proposition}
\label{prop-Dprime-1}
Cases \xref{m5-b10} and \xref{m11-b11} do not occur.
In particular,
the assertion of Theorem \xref{main} holds if $m_o$ is prime.
\end{proposition}
\begin{proof}

Consider the case \ref{m11-b11}. Since $\gcd(q,m)=1$,
there is $0<l<r=m$ such that $ql\equiv 1 \mod m$.
Then by \eqref{eq-RR-e3} and \eqref{eq-RR-cPK-K} we have 
\[
-\frac l{11}=c_P(lD)=c_P(K)=-\frac {r^2-1}{12r}=-\frac {10}{11},
\]
so $l=q=10$. Then again by \eqref{eq-RR-e3} and \eqref{eq-RR-cPK-K}
\[
-\frac 1{11} =c_P(D)=c_P(-K)=\frac {r^2-1}{12r}-\frac{b(r-b)}{2r}=
\frac {10}{11}-\frac{b(11-b)}{22}.
\]
Hence, $b(11-b)=22$ and $b$ cannot be coprime to $11$, a contradiction.

Consider the case \ref{m5-b10}.
Since $mq=5q\equiv 0\mod r=10$, $q$ is even.
There is $0<l<5$ such that $ql\equiv 2 \mod r$.
Then by \eqref{eq-RR-e3} we have 
\[
-\frac l{5}=c_P(lD)=c_P(2K)=-\frac {2(r^2-1)}{12r}+\frac{b(r-b)}{2r}
=\frac {b(10-b)-33}{20}.
\]
Thus $b(10-b)+4l=33$, $b\in \{3, \, 7\}$, $l=3$, and $q=4$.
Again by \eqref{eq-RR-e3} 
\begin{equation*}
-\frac 15=c_P(D)=-\frac {4(r^2-1)}{12r}+
\sum_{j=1}^3\frac{\ov{bj}(r-\ov{bj})}{2r}=
-\frac {33}{10}+
\sum_{j=1}^3\frac{\ov{3j}(10-\ov{3j})}{20}=
-\frac {3}{5},
\end{equation*}
a contradiction. This proves our lemma.
\end{proof}

\begin{corollary}
\label{cor-p}
For every prime divisor $d$ of $m_o$ 
we have $d\in \{2,\, 3,\, 5\}$.
\end{corollary}
\begin{proof}
Apply Propositions \ref{prop-Dprime} and \ref{prop-Dprime-1} with $D=\frac{m_o}{d}F_o$.
\end{proof}

Let $P_i$ be points of $\B(F_o)$.
Let $P=P_1$ be a point in $\B(F_o)$
whose index $r_{P_1}$ is divisible by $m_o$ (see Proposition \ref{cor-point}).
For short, below we will write $r_i$, $b_i$, $q_i$, etc
instead of $r_{P_i}$, $b_{P_i}$, $q_{P_i}$, respectively.

\begin{corollary}
$m_o$ is not divisible by $m \in \{16, 27, 25,
10,15, 12, 18\}$.
\end{corollary}

\begin{proof}
Let $d=2$, $3$ or $5$ be a prime divisor of $m_o$ 
and let $D=\frac {m_o}{d}F_o$.
Then $dD=f^*(o)$ and $D$ 
is not Cartier at $P_1$.
In this case, by 
Propositions \ref{prop-Dprime} and \ref{prop-Dprime-1} the index of $(X,P_1)$
is at most $9$, a contradiction.
\end{proof}

\begin{corollary}
If $m_o$ is not prime, then $m_o \in \{4,\, 6,\, 8,\, 9\}$.
\end{corollary}

\begin{lemma}
\label{main-m6-b236}
If $m_o= 6$, then $\B(F_o)=\basket{2, 3, 6}$.
Moreover, $\gcd(r_P,q_P)=1$ for all $P\in \B(F_o)$.
\end{lemma}

\begin{proof}
Take $D=3F_o$.
Then $2D\sim f^*(o)$ but $D$ is not Cartier at $P_1$.
Hence $(X,P_1)$ is of index 6 and for $D$
we are in the case \ref{main-m2-b26},
that is, $\B(3F_0)=\basket{2,\, 6}$.
At all points $P_i\notin \B(3F_0)$ the divisor $3F_o$ 
is Cartier.
Similarly, take $D'=2F_o$. Then for
$D$ we get the case \ref{main-m3-b36},
that is, $\B(2F_o)=\basket{3,\, 6}$.
Hence $\B(F_o)$ contains three points 
$P_1$, $P_2$, $P_3$ of indices $6$, $2$, $3$,
respectively, and in all other points 
both $D'=2F_o$ and $D=3F_o$ are Cartier.
Hence $F_o=D-D'$ is Cartier outside of $P_1$, $P_2$, $P_3$
and $\B(F_o)=\basket{2, 3, 6}$.
\end{proof}

\begin{lemma}
\label{main-m4-b244}
If $m_o= 4$, then $\B(F_o)=\basket{2, 4, 4}$.
Moreover, $\gcd(r_P,q_P)=1$ for all $P\in \B(F_o)$.
\end{lemma}

\begin{proof}
Clearly, $2F_o$ is Cartier at all points of index $2$.
Hence $\B(2F_o)$ contains no such points
and for $\B(2F_o)$ we are in the case \ref{main-m2-b8}
or \ref{main-m2-b44}.
For all points $P_i\notin \B(2F_o)$
the divisor $2F_o$ is Cartier at $P_i$.
Hence, $q_{i}=r_{i}/2$.

Assume that $\B(2F_o)=\basket{8}$.
Let $P\in \B(2F_o)$. 
Since $4F_o$ is Cartier, $4q_P\equiv 0 \mod 8$
(but $2q_P\not \equiv 0 \mod 8$).
By Lemma \ref{lemma-Sp} and \ref{lemma-su-div} we have 
\[
\Xi_{P_1,4}=-\frac{5}{4},
\qquad \Xi_{P_j,8}=4\Xi_{P_j,2}=-\frac {r_{j}}4 , 
\quad j\neq 1.
\]
Therefore, by \eqref{eq-RR-su} the following holds $\sum_{i\neq 1} r_i=1$,
a contradiction.

Hence $\B(2F_o)=\basket{4,4}$. 
At both points $P_i\in \B(2F_o)$
we have $F_o\sim \pm K_X$ near $P_i$.
Again by Lemma \ref{lemma-Sp} and \ref{lemma-su-div} 
\[
\Xi_{P_i,4}=-\frac{5}{8},\quad i=1,\, 2
\qquad \Xi_{P_j,4}=2\Xi_{P_j,2}=-\frac {r_{j}}8, 
\quad j\neq 1, \, 2.
\]
Therefore, by \eqref{eq-RR-su} we have $\sum\limits_{i\neq 1,\, 2} r_i=2$
and there is only one solution $\B(F_o)=\basket{4,4,2}$.
\end{proof}

\begin{corollary}
$m_o\neq 8$
\end{corollary}
\begin{proof}
Indeed, if $m_o=8$, then for $\B(2F_o)$
there is only one possibility from Lemma \ref{main-m4-b244}.
This contradicts Proposition \ref{cor-point}.
\end{proof}

\begin{lemma}
$m_o\neq 9$.
\end{lemma}
\begin{proof}
Assume that $m_o=9$. Take $D:=3F_o$.
Then $3D\sim f^*(o)$ but $D$ is not Cartier at $P_1$.
Hence, $\gcd(q_{1},r_{1})=1$, $(X,P_1)$ is of index $9$ and for $D$
we are in the case \ref{main-m3-b9},
that is, $\B(D)=\basket{9}\subset \B(F_o)$.
In all points $P_i\in \B(F_o)$, $P_i\neq P_1$ the divisor $D=3F_o$ 
is Cartier.
Hence by Lemma \ref{lemma-Sp} and \ref{lemma-su-div} we have 
\[
\Xi_{P_1,9}=-\frac{10}{3},
\qquad \Xi_{P_i,9}=3\Xi_{P_i,3}=-\frac {r_{i}}3, \quad i\neq 1.
\]
Therefore, by \eqref{eq-RR-su}
\[
-4=\sum \Xi_{P_i,m}=-\frac{10}{3}-\frac 13 \sum_{i\neq 1} r_{i},
\qquad r_{i}= 2.
\]
This contradicts $r_{i}\equiv 0\mod 3$.
\end{proof}

\begin{computation}
The last lemma finishes the proof of Theorem \ref{main}.
It remains to compute values $b_k$, $q_k$, and $K_{F_g}^2$
in Table \ref{table}.

First we compute the possible values of $q_i$.
We may assume that $1\le q_i<r_i$.
In regular cases (\type{2,3,6}, \type{5,5}, \type{3,3, 3},
\type{2,4,4}, \type{2,2,2,2}) we have $\gcd (q_i, r_i)=1$
(see Lemmas \ref{main-m6-b236} and \ref{main-m4-b244})
and $m_o\ge r_i$ for all $i$.
Take $1\le l\le m_o-1$ so that $q_il\equiv 1\mod r_i$.
Then by \eqref{eq-RR-cPK-K} and \eqref{eq-RR-e3}
the following equality holds
\[
\sum_i c_{P_i}(lF_o)=\sum_i c_{P_i}(K)=
-\sum_i \frac{r_i^2-1}{12r_i}=-\frac{l}{m_o}.
\]
From this we immediately obtain $l\equiv q_i\equiv -1 \mod r_i$
for all $i$.

If $m_o=2$ (cases \type{4\times 2},
\type{2,2,4}, \type{4,4}, \type{2,6}, 
\type{8}), then $2F_o$ is Cartier.
Hence $q_i=r_i/2$. It remains to consider only cases
\type{9} and \type{3,6}.
In case \type{9}, since $3F_o$ is Cartier, 
we have $q:=q_1=3$ or $6$.
If $q=3$, then by \eqref{eq-RR-e3} we have
\[
-1=3c_P(F_o)=3c_P(3K)=- \frac {40}{6}
+\frac{b(9-b)}{6}+\frac{\ov{2b}(9-\ov{2b})}{6}. 
\]
Hence, $34=b(9-b)+\ov{2b}(9-\ov{2b})$ and 
$5b^2\equiv 2\mod 9$. This immediately implies $b\equiv \pm 2$.
Similarly, if $q=6$, then $b^2\equiv -2\mod 9$
and $b\equiv \pm 4$.

Finally consider the case \type{3,6}.
Then by \eqref{eq-RR-cPK-K} and \eqref{eq-RR-cP-def}
\[
c_{P_1}(F_o)=
\begin{cases}
-2/9&\text{if $q_1=1$} 
\\
-1/9&\text{if $q_1=2$} 
\end{cases}
\qquad
c_{P_2}(F_o)=
\begin{cases}
-5/9&\text{if $q_1=2$} 
\\
-1/9&\text{if $q_1=4$} 
\end{cases}
\]
The equality $c_{P_1}(F_o)+c_{P_2}(F_o)=-1/3$
(see \eqref{eq-RR-e3}) holds only if $q_1=1$, $q_2=4$. 

\begin{corollary}
The fiber $F_o$ is regular if and only if $q_i\equiv -1\mod r_i$
for all $i$. In particular, 
for regular $F_o$ near each point
$P\in F_o$ 
where $F_o$ is not Cartier we have $K_X+F_o\sim 0$ .
\end{corollary}

\end{computation}

\begin{computation}
Now we find the possible values of $b_i$.
In all cases except for \type{5,5} and \type{9}
the relations $\gcd(r_i,q_i)=1$ is sufficient to get the conclusion.
The case \type{9} was treated above.
Consider the case \type{5,5}. Then by \eqref{eq-RR-cPK-K} and 
\eqref{eq-RR-e3} we have $10=b_1(5-b_1)+b_2(5-b_2)$.
Hence $b_1^2+b_2^2\equiv 0\mod 5$.
\end{computation}

\begin{computation}
To obtain the possible values for $K_{F_g}^2$
we use \eqref{eq-Delta} with $a=0$.
Since $\Delta_a$ is an integer, it is sufficient to 
compute $\delta_0=c_P(-K)-c_P(-K-F_o)$. 
Table \ref{table1}
gives all values of $\delta_0$.
For example, if $F_o$ is regular, then
$q_P\equiv -1\mod r_P$ for all $P$ and 
$\delta_0=\sum c_P(-K)=\sum c_P(F_o)$. So by 
\eqref{eq-RR-cPK-K} and \eqref{eq-RR-e3}
we have $\delta_0=-1/m_o$.
Assume that $q_P=r_P/2$ (and all the $r_p$ are even). Then 
\[
\delta_0 =\sum_{P\in \B}\left( c_P(-K)- c_P\left(\frac{r_P-2}{2}K\right)\right).
\]
Hence by \eqref{eq-RR-cPK-K} and \eqref{eq-RR-cP-def}
\begin{itemize}
\item[]
$r_P=2$ $\Longrightarrow$ $\delta_0= c_P(-K)= -1/8$,
\item []
$r_P=4$ $\Longrightarrow$ $\delta_0= c_P(-K)- c_P(K)= 1/4$,
\item []
$r_P=6$ $\Longrightarrow$ $\delta_0= c_P(-K)- c_P(2K)= 5/8$,
\item []
$r_P=8$ $\Longrightarrow$ $\delta_0= c_P(-K)- c_P(3K)= 
1$ or $0$ if $b_P=1$ or $3$, respectively.
\end{itemize}
This immediately gives the values of $\delta_0$ in cases \type{2,2,4},
\type{4,4}, \type{2,6}, and \type{8}. Cases \type{3,6} and \type{9}
are similar.
\end{computation}

\begin{table}[h]
\renewcommand{\extrarowheight}{5pt}
\renewcommand{\arraystretch}{1.3}
\caption{}
\label{table1}
\begin{tabular}{|p{50pt}||c|c|c|c|c|c|c|c|c|c|c|}
\hline
 &regular &\type{3,6}
&\type{9}
&\type{2,2,4}&\type{4,4}&\type{2,6}
&\type{8}
\\[5pt]
\hline
$\delta_0$&$-\frac1{m_o}$
&$\frac23$& $\frac{6-q_1}9$&
$0$&$\frac12$&$\frac12$&$\frac{3-|b_1|}{2}$
\\[5pt] \hline
\end{tabular}
\end{table}

\section{Examples}
\label{sect-exam}
In this section we construct some examples of del Pezzo bundles
with multiple fibers. We 
use notation of Construction \ref{base-change}.
We start with regular case.
\begin{proposition}
\label{prop-examples}
Let $f'\colon X'\to Z'\ni o'$ be a Gorenstein del Pezzo bundle.
Assume that the central fiber $F_o':=f'^{-1}(o')$
has only Du Val singularities.
Assume also that the cyclic group 
$\muu_{m_o}$ acts on $X'$ and $Z'$ so that
\begin{enumerate}
\item 
the action on $Z'$ is free outside of $o'$,
\item
$f'$ is $\muu_{m_o}$-equivariant, 
\item 
the action on $F_o'$ is free in codimension one,
\item 
the quotient $F_o:=F_o'/\muu_{m_o}$ has only
Du Val singularities.
\end{enumerate}
Then $f\colon X=X'/\muu_{m_o}\to Z=Z'/\muu_{m_o}$
is a del Pezzo bundle with regular central fiber of multiplicity
$m_o$ and, moreover, $F_o\sim -K_X$ near each point $P\in X$.
\end{proposition}
\begin{proof}
In notation of Construction \ref{base-change} it is sufficient to 
show that $X$ has only terminal singularities.
Since $X'$ has only terminal singularities and
the action of $\muu_{m_o}$ is free outside of 
a finite number of points $P'_k$ lying on $F_o'$, 
the quotient $X$ is smooth outside of $\pi(P_k')\in F_o$.
By the inversion of adjunction \cite[17.6]{Utah} the
pair $(X,F_o)$ is PLT near $F_o$. Since $F_o$ is Gorenstein,
the divisor $K_X+F_o$ is Cartier. Hence the
pair $(X,F_o)$ is canonical near $F_o$ and so
$X$ has only terminal singularities.
\end{proof}

Now we apply Proposition \ref{prop-examples} to construct concrete
examples.
\begin{example}
Let $F_o'$ be a del Pezzo surface of degree $d:=K_{F_o'}^2$
with at worst Du Val singularities.
Assume that the group $\muu_{m_o}$, $m_o\ge 2$ 
acts on $F_o'$ freely in codimension one 
and so that the quotient $F_o:=F_o'/\muu_{m_o}$ has again only
Du Val singularities. Clearly, $F_o$ 
is del Pezzo surface and $m_oK_{F_o}^2=d$.
Hence, $d\ge m_o\ge 2$. 
For $d=2$, $3$, $4$, and $8$, 
according to \cite{Hidaka1981} there is an embedding
\[
\begin{array}{llll}
F_o'&\subset& \PP:=\PP(1,1,1,2) &\text{if $d=2$}
\\[5pt]
F_o'&\subset& \PP:=\PP^3 &\text{if $d=3$}
\\[5pt]
F_o'&\subset& \PP:=\PP^4 &\text{if $d=4$}
\\[5pt]
F_o'&\subset& \PP:=\PP^3\quad &\text{if $d=8$}
\end{array}
\]
Moreover, if $d=2$, $3$, $8$, then $F_o$ is a (weighted) hypersurface 
of degree $4$, $3$, $2$, respectively and if $d=4$, then $F_o'$
is an intersection of two quadrics. 
The action of $\muu_{m_o}$ on $F_o'$ induces the action on $\PP$.
We fix a linearization of this action 
and take semi-invariant coordinates $x_i$ in $\PP$.
Now we define $\muu_{m_o}$-equivariant del Pezzo bundle $f'\colon X'\to Z'$.
If $F_o'$ is smooth, we can take $X'=F_o'\times \CC_t$.
In general case, $X'$ is embedded into $\PP\times \CC_t$, $Z'=\CC_t$
and $f'$ is the projection, where 
$t$ is a coordinate in $\CC$ with $\wt t=1$. 
Consider for example the case 
$d\le 3$ (case $d=4$ is similar). 
Let $\phi=\phi(x_1,x_2,x_3,x_4)$ 
be the defining 
equation of $F_o'$ and let $\gamma_k$ be all monomials of weighted degree 
$d$. 
For each $\gamma_k$, let 
$n_k$ be the smallest positive integer such that
$ n_k \equiv -\wt \gamma_k \mod m_o$. Then the polynomial
$\psi(x_1,\dots,x_4; t):=\phi+ \sum c_k t^{n_k} \gamma_k$,\ $c_k\in \CC$ 
is $\muu_{m_o}$-semi-invariant. Let $X'=\{\psi=0\}\subset \PP\times \CC_t$.
By Betrtini's theorem, for sufficiently general constants $c_k$, 
fibers $F_t'$ of $f'$ over $t\neq 0$ are 
smooth del Pezzo surfaces. 
Hence we can apply Proposition \ref{prop-examples}
and get a del Pezzo bundle with a regular fiber of multiplicity $m_o$.

Note that the map $F_o'\to F_o$ is \'etale outside of $\Sing F_{o}$.
Hence there is a surjection $\pi_1(F_o\setminus \Sing F_{o})\twoheadrightarrow \muu_{m_o}$.
Conversely, assume that $F_o$ is a del Pezzo surface with Du Val singularities 
such that $\pi_1(F_o\setminus \Sing F_{o})\twoheadrightarrow \muu_{m_o}$.
Then there is an \'etale outside of $\Sing F_o$ cyclic $\muu_{m_o}$-cover 
$\upsilon\colon F_o'\to F_o$. Since $K_{F_o'}=\upsilon^*K_{F_o}$,
$F_o'$ is also a del Pezzo surface with Du Val singularities.
The fundamental groups of smooth loci of Du Val del Pezzo surfaces 
are described in \cite{Miyanishi-Zhang-1988}, \cite{Miyanishi-Zhang-1993}.
For example, from \cite{Miyanishi-Zhang-1988} we have the following 
examples with $\rho(F_o)=1$ (we do not list all the possibilities):
\par\medskip\noindent
\setlongtables\renewcommand{\arraystretch}{1.3}
\begin{longtable}{|c|l|c|c|c|c|l|}
\hline
$K_{F_o}^2$&$\Sing F_o$&$m_o$& $K_{F_o'}^2=K_{F_g}^2$&$\rho(F_o')$&$F_o'$, \ $\Sing F_o'$&type
\\[5pt]
\hline
\endfirsthead
\hline
$K_{F_o}^2$&$\Sing F_o$&$m_o$& $K_{F_o'}^2=K_{F_g}^2$&$\rho(F_o')$&$F_o'$, \ $\Sing F_o'$&type
\\[5pt]
\hline
\endhead
\hline
\endlastfoot
\hline
\endfoot
$1$& $A_1A_2A_5$&$6$&$6$&$4$&smooth&\type{2,3,6}
\\
$1$& $2A_4$&$5$&$5$&5&smooth&\type{5,5}
\\
2&$A_12A_3$&4&8&2&$\PP^1\times \PP^1$&\type{2,4,4}
\\
1&$A_3D_5$&4&4&4&$A_2$&\type{2,4,4}
\\
3&$3A_2$&3&9&1&$\PP^2$&\type{3,3,3}
\\
2&$A_2A_5$&3&6&3&$A_1$&\type{3,3,3}
\\
1&$A_8$&3&3&5&$A_2$&\type{3,3,3}

\\
4&$2A_1A_3$&2&8&1&$\PP(1,1,2)$&\type{2,2,2,2}
\\
3&$A_1A_5$&2&6&2&$A_2$&\type{2,2,2,2}
\\
2&$A_7$&2&4&3&$A_3$&\type{2,2,2,2}
\\
1&$D_8$&2&2&3&$D_5$&\type{2,2,2,2}
\end{longtable}
\end{example}

\begin{example}
In some cases we can give more explicit construction.
As was mentioned above, if $ F_o'$ is smooth, we can 
take $X'=Z'\times F_o'$.
Consider the following cases:
\begin{itemize}
\item 
$F_o'=\PP^2$, $\muu_3$ acts on $\PP^2_{x,y}$
by $x\mapsto \epsilon x$, $y\mapsto \epsilon^{-1} y$
(here $x$, $y$ are non-homogeneous coordinates on $\PP^2$ and $\epsilon^3=1$).
Then $\PP^2/\muu_3$ is a toric del Pezzo surface of degree 3 
having three
singular points of type $A_2$.
The quotient $f\colon X\to Z$ is a del Pezzo bundle 
with special fiber of type \type{3,3,3}. 
\item 
$F_o'=\PP^1\times \PP^1$, $\muu_2$ acts on $\PP^1_x\times \PP^1_y$
by $x\mapsto - x$, $y\mapsto - y$.
Then $\PP^1\times \PP^1/\muu_2$ 
is a del Pezzo surface of degree 4 
having four
singular points of type $A_1$.
The quotient $f\colon X\to Z$ is a del Pezzo bundle 
with special fiber of type \type{2,2,2,2}. 
\item 
$F_o'=\PP^1\times \PP^1$, $\muu_4$ acts 
by $x\mapsto y$, $y\mapsto - x$.
Then $\PP^1\times \PP^1/\muu_4$ 
is a del Pezzo surface of degree 2 
having two points of type $A_3$ and one point of type $A_1$.
The quotient $f\colon X\to Z$ is a del Pezzo bundle 
with special fiber of type \type{2,4,4}. 
\end{itemize}
\end{example}

Now we give some examples of irregular multiple fibers.
\begin{example}
Recall that any smooth del Pezzo surface of degree $1$ 
can be realized as a weighted hypersurface of degree $6$ in $\PP=\PP(1,1,2,3)$. 
Let \[
\phi(x_1,x_2,y,z)=a_1x_1^6+a_2x_2^6+y^2(b_1x_1^2+b_2x_2^2) +c z^2,\qquad a_i, b_j, c\in \CC^*
\]
be a polynomial of weighted degree $6$,
where $x_1$, $x_2$, $y$, $z$ are coordinates in $\PP$ with 
$\wt x_i=1$, $\wt y=2$, $\wt z=3$. 
Consider the hypersurface $F_o'\subset \PP$
given by $\phi=0$.
By Bertini's theorem, for sufficiently general $a_i, b_j, c$, the surface 
$F_o'$ is smooth outside of $P':=(0:0:1:0)$.
Consider the subvariety $X'$ in $\PP\times \CC_t$ given by 
$\phi+ty^3=0$ and let $f'\colon X'\to Z'=\CC$ be the natural projection.
Since $F_o'$ is the scheme fiber of the projection 
$f'\colon X'\to Z'$, the variety $X'$ is smooth outside of $P'$.
We identify $F_o'$ with the fiber over $t=0$. 
Then $f'$ is a del Pezzo bundle of degree $1$ having a unique 
singular point of type $\frac12(1,1,1)$ at $P'$.

Now let $\muu_2$ acts on $\PP\times \CC$ and $X'$ by
\[
(x_1,x_2,y,z; t) \longmapsto (x_1,-x_2,-y,-z; -t).
\]
The locus of fixed points $\Lambda$ consists of 
the line $L:=\{x_1=y=t=0\}$ and two isolated points 
$P':=(0:0:1:0; 0)$ and $P_1:=(1:0:0:0; 0)$.
Then $F_o'\cap \Lambda=\{P', Q_1, Q_2\}$, where
$Q_1\neq Q_2$ are points given by $x_1=y=a_2x_2^6+z^2=t=0$.
Let $f\colon X=X'/\muu_2\to Z=Z'/\muu_2$
be the quotient of $f'$. Since the action of $\muu_2$ on $X'$
is free in codimension one, $-K_X$ is $f$-ample and $F_o:=F_o'/\muu_2$
is a fiber of multiplicity $2$.
We show that $X$ has only terminal singularities.
By the above, $X$ is smooth outside of images of $P'$, $Q_1$, $Q_2$.
Since the $(X',Q_i)$ are smooth points, quotients $(X',Q_i)/\muu_2$
are terminal of type $\frac12(1,1,1)$.
Consider the affine chart 
$\{y\neq 0\}\simeq \CC^4_{x_1', x_2',z',t}/\muu_2(1,1,1,0)$
containing $P'$.
Here $X'$ is given by the equation $\phi(x_1',x_2',1,z')+t=0$ 
and the action of $\muu_2$ on $\PP$ induces the following action of 
$\muu_4$:
\[
(x_1',x_2',z',t) \longmapsto (\mt{i}x_1',-\mt{i}x_2',\mt{i}z',-t),\quad \mt{i}=\sqrt{-1}.
\]
Thus the quotients $(X',P')/\muu_2$ is a 
terminal cyclic quotient of type $\frac14(1,-1,1)$.
Therefore, $f\colon X\to Z$ is a del Pezzo bundle 
with special fiber of type \type{2,2,4}.
\end{example}

\begin{example}
As above let $\PP=\PP(1,1,2,3)$ and let 
\[
\phi(x_1,x_2,y,z)=a_1x_1^6+a_2x_2^6+c y^3,\qquad a_i, c\in \CC^*
\]
be a $\muu_2$-invariant polynomial of weighted degree $6$. 
Consider the hypersurface $F_o'\subset \PP$
given by $\phi=0$.
Again for sufficiently general $a_i, c$, the surface 
$F_o'$ is smooth outside of $P'':=(0:0:0:1)$.
Consider the subvariety $X'$ in $\PP\times \CC_t$ given by 
$\phi+tz^2=0$ and let $f'\colon X'\to Z'=\CC$ be the natural projection.
Then $f'$ is a del Pezzo bundle of degree $1$ having a unique 
singular point of type $\frac13(1,1,-1)$ at $P''$.
Now let $\muu_3$ acts on $\PP\times \CC$ and $X'$ by
\[
(x_1,x_2,y,z; t) \longmapsto (x_1,\epsilon x_2,\epsilon y,\epsilon z; \epsilon t),
\qquad \epsilon :=\exp(2\pi\mt{i}/3).
\]
The only fixed point on $X'$ is $P''$. 
As above, one can check that $(X',P'')/\muu_3$
is a terminal point of type $\frac19(-1,2,1)$.
Therefore, $X/\muu_3\to Z'/\muu_3$ is a del Pezzo bundle 
with special fiber of type \type{9}.
\end{example}

\begin{example}
\label{ex-not-basket}
Let $\PP:=\PP(1,1,1,2,2)$, let $x_1,x_2,x_3,y_1,y_2$
be coordinates,
and let $X'\subset \PP\times \CC$ be subvariety given by
\[
\begin{cases}
c_1y_1^2+c_2y_2^2&=a_1x_1^4+a_2x_2^4+a_3x_3^4
\\
ty_2&=b_1x_1^2+b_2x_2^2+b_3x_3^2,
\end{cases}
\]
where $t$ is a coordinate on $\CC$ and $a_i$, $b_j$, $c_k$ 
are sufficiently general constants. By Bertini's theorem 
$X'$ is smooth outside of $\{x_1=x_2=x_3=0\}\subset \Sing \PP$.
It is easy to check that $X'\cap \Sing \PP$ consists of two points
\[
\{P_1',\, P_2'\}=\{t=x_1=x_2=x_3=0,\ c_1y_1^2+c_2y_2^2=0\}
\]
and these points are 
terminal of type $\frac12(1,1,1)$. 
The projection $X'\to \CC$ is a del Pezzo bundle of degree $2$.
Define the action of $\muu_2$ by
\[
(x_1,x_2,x_3,y_1,y_2; t) \longmapsto (x_1,x_2,-x_3,y_1,-y_2; -t).
\]
There are four fixed points 
\[
\{Q_1',\dots, Q_4'\}=\{t=x_3=y_2=0,\ c_1y_1^2=a_1x_1^4+a_2x_2^4, \ b_1x_1^2+b_2x_2^2=0\}.
\]
The quotient $f\colon X'/\muu_2\to \CC/\muu_2$ is a del Pezzo bundle 
of type \type{2,2,2,2}. Note however that the image $P$ of 
$\{P_1',\, P_2'\}$ on $X'/\muu_2$ is a point of type $\frac12(1,1,1)$
and $F_o$ is Cartier at $P$ (i.e., $P\notin \B(F_o)$). 
\end{example}
\begin{example}
In the above notation define another action of $\muu_2$:
\[
(x_1,x_2,x_3,y_1,y_2; t) \longmapsto (x_1,x_2,-x_3,-y_1,-y_2; -t).
\]
Then the quotient $f\colon X'/\muu_2\to \CC/\muu_2$ is a del Pezzo bundle 
of type \type{4,4}. 
\end{example}

\begin{example}
\label{ex-I36}
Let $\PP:=\PP(1,1,1,1,2)$, let $x_1,x_2,x_3,x_4,y$
be coordinates,
and let $X'\subset \PP\times \CC$ be subvariety given by
\[
\begin{cases}
a_1x_1^2+a_2x_2^2+a_3x_1x_2+a_4x_3x_4=ty
\\
b_1x_1^3+b_2x_2^3+b_3x_3^3=x_4y
\end{cases}
\]
where $t$ is a coordinate on $\CC$ and $a_i$, $b_j$
are sufficiently general constants. Then the variety
$X'$ is smooth outside of the point $P'=\{x_1=x_2=x_3=x_4=0\}$
and $P'\in X'$ is of type $\frac12(1,1,1)$.
The projection $X'\to \CC$ is a del Pezzo bundle of degree $3$.
Define the action of $\muu_3$ by
\[
(x_1,x_2,x_3,x_4,y; t) \longmapsto 
(\omega^{-1}x_1,\omega^{-1}x_2,\omega^{}x_3,x_4,y; \omega^{}t).
\]
There are two fixed points $\{t=x_1=x_2=x_3=x_4y=0\}$
and quotients of these points are of types $\frac16(1,1,-1)$ and 
$\frac13(,1,-1)$.
Hence the quotient $f\colon X'/\muu_3\to \CC/\muu_3$ is a del Pezzo bundle 
of type \type{3,6}. 
\end{example}

\section{On del Pezzo bundles with fibers of multiplicity $\ge 5$.}
\label{sect-45}
\begin{notation}
\label{not-56}
Let $f\colon X\to Z\ni o$ be the germ of a del Pezzo bundle
and let $m_oF_o=f^*(o)$ be a fiber of 
multiplicity $m_o$. 
In this section we assume that $m_0\ge 5$, i.e., $F_o$ is of type \type{2,3,6}
or \type{5,5}.
\end{notation}

\begin{conjecture}
\label{conj-non-Gor}
In notation of \xref{not-56}
$f$ is a quotient
of a Gorenstein del Pezzo bundle by a cyclic
group acting free in codimension $2$ on $X$.
\end{conjecture}

\begin{proposition}
\label{prop-non-Gor}
Notation as in \xref{not-56}.
If either
\begin{enumerate}
\item
$\B(F_o)=\B$, that is, each point $P\in F_o$ 
where $F_o$ is Cartier is Gorenstein on $X$, or
\item
a general member $S\in |-K_X|$ has only Du Val singularities
\textup(Reid's general elephant conjecture\textup),
\end{enumerate}
then \xref{conj-non-Gor} holds.
\end{proposition}

\begin{proof}
Assume that (i) holds.
By Table \ref{table} 
near each singular point
$K_X+F_o\sim 0$. 
Apply Construction \ref{base-change}.
Then $F_o'=\pi^*F_o$ is Cartier.
Since $\pi$ is \'etale in codimension one, $K_{X'}+F_o'\sim 0$.
Hence, $X'$ is Gorenstein.

Now assume that (ii) holds.
Then $\varphi\colon S\to Z$ is an elliptic fibration with 
Du Val singularities. We have $K_S=(K_X+S)|_S\sim 0$.
Let $\mu \colon \tilde S\to S$ be the minimal resolution.
Since $S$ has only Du Val singularities, $K_{\tilde S}\sim 0$.
In particular, $\psi\colon \tilde S\to Z$ is a minimal elliptic fibration.
By Kodaira's canonical bundle formula 
$\psi$ has no multiple fibers \cite[Th. 12]{Kodaira-1964}.
Since $\psi^*o$ has a component of multiplicity $\ge 5$,
for $\psi^*o$ we have only one possibility $\tilde E_{8}$
in the classification of singular fibers
\cite[Th. 6.2]{Kodaira-1963}.
More precisely, $\Supp (\psi^*o)$ is a tree 
of smooth rational curves with self-intersection number $-2$
and the dual graph $\Gamma$ is the following:
\[
\label{eq-graph-tildeE8}
\xymatrix@R=0.8pc{
\stackrel{1}\circ\ar@{-}[r]
&\stackrel{2}\circ\ar@{-}[r]
& \stackrel{3}\circ\ar@{-}[r] 
&\stackrel{4}\circ\ar@{-}[r]
&\stackrel{5}\circ\ar@{-}[r]
&\stackrel{6}\circ\ar@{-}[r]
&\stackrel{4}\circ\ar@{-}[r]
&\stackrel{2}\circ
\\
&&&&&\stackunder{3}\circ\ar@{-}[u]
}
\]
Further we consider the case $m_o=6$ (the case $m_o=5$ is similar).
It is easy to see that the curve $S\cap F_o$ is irreducible and 
correspond to the central vertex $v$ of $\Gamma$.
Then $\Gamma \setminus \{v\}$ has three connected components corresponding to 
points of types $A_1$, $A_2$ and $A_5$ on $S$.
Therefore, $\B(F_o)=\B$.
\end{proof}

\begin{proposition}
In notation of \xref{not-56}, assume that $F_o$ is irreducible.
let $f_{\an}\colon X_{\an}\to Z_{\an}$
be the analytic germ near $F_o$.
Then
$X_{\an}$ is $\QQ$-factorial over $Z_{\an}$,
$\rho (X_{\an}/ Z_{\an})=1$, and $\rho(F_o)=1$.
\end{proposition}

\begin{warning}
Here the $\QQ$-factoriality condition of
$X_{\an}$ means 
that every global Weil divisor of the total germ $X_{\an}$ along
$F_o$ is $\QQ$-Cartier, not that every analytic local ring of $X_{\an}$
is $\QQ$-factorial.
\end{warning}

\begin{proof}
Let $q\colon \hat X_{\an}\to X_{\an}$ be a $\QQ$-factorialization over 
$Z_{\an}$. Run the MMP over $Z_{\an}$. So, we have the following diagram
\[
\xymatrix{
&\hat X_{\an}\ar[dl]_{q}\ar@{-->}[dr]&
\\
X_{\an}\ar[d]_{f_{\an}}&&\bar X_{\an}\ar[d]^{\bar f_{\an}}
\\
Z_{\an}&&\bar Z_{\an}\ar[ll]_{g_{\an}}
}
\]
Here $\bar X_{\an}$ is $\QQ$-factorial over $\bar Z_{\an}$ and
$\rho (\bar X_{\an}/ \bar Z_{\an})=1$.
Note that $\hat X_{\an} \dashrightarrow \bar X_{\an}$ 
is a composition of flips and divisorial contractions that 
contract divisors to curves dominating $Z_{\an}$.
Let $\bar F_o$ be the proper transform of $F_o$ on $\bar X_{\an}$.
There are two possibilities:

1) $\bar Z_{\an}$ is a surface. Then $g_{\an}$ is a rational curve fibration 
with $\rho(\bar Z_{\an}/Z_{\an})=1$. Let $C:=\bar f_{\an} (\bar F_o)$.
Since $\bar X_{\an}$ has only isolated singularities, $\bar F_o=\bar f_{\an}^*(C)$.
Further, $g_{\an}^*(o)=nC$ for some $n\in \ZZ_{>0}$ and 
$\bar f_{\an}^*g_{\an}^*(o)=n\bar f_{\an}^*C=n\bar F_o$.
So, $n=m_o$.
By the main result of \cite{Mori-Prokhorov-2008}
the surface $\bar Z_{\an}$ has only Du Val singularities.
Therefore, $m_o=n\le 2$, a contradiction. 

2) $\bar Z_{\an}$ is a curve. Then $g_{\an}$ is an isomorphism
and $\bar f_{\an}\colon \bar X_{\an}\to \bar Z_{\an}$ is a del Pezzo bundle
with central fiber $\bar F_o$ of multiplicity $\bar m_o=m_o\ge 5$.
By Table \ref{table} the degree of the generic fiber of $\bar f_{\an}$ 
(and $f_{\an}$)
is equal to $m_o$.
This means that degrees of generic fibers of $\bar f_{\an}$
and $f_{\an}$ coincide. In particular, the MMP $\hat X_{\an}\dashrightarrow \bar X_{\an}$
does not contract any divisors. Hence, 
$\rho(\hat X_{\an}/Z_{\an}) =\rho(\bar X_{\an}/Z_{\an})=1$.
This implies that $q$ is an isomorphism and $\rho(X_{\an}/Z_{\an})=1$.
The last assertion follows from the exponential exact sequence and vanishing 
$R^1 f_{\an *}\OOO_{X_{\an}}=0$. 
\end{proof}

\begin{proposition}
Notation as in \xref{not-56}.
Conjecture \xref{conj-non-Gor}
holds under the additional assumption that $F_o$ has only log terminal singularities.
\end{proposition}
\begin{proof}
Assume that $F_o$ has only log terminal singularities.
By Table \ref{table} near each point $P\in \B(F_o)$ we have $K_X+F_o\sim 0$.
By Adjunction $F_o$ has only Du Val singularities at these points.
In points $P\notin \B(F_o)$ the divisor $F_o$ is Cartier.
Hence $F_o$ 
has only singularities of type T \cite{Kollar-ShB-1988}.
By Noether's formula \cite[Prop. 3.5]{Hacking-Prokhorov-2005}
\[
K_{F_o}^2+\rho(F_o)+\sum_{P\in F_o} \mu_P=10.
\]
Since points in
$\B(F_o)$ correspond to distinct points on 
$X$, we have $\sum_{P\in \B(F_o)} \mu_P\ge 8$.
Hence, $K_{F_o}^2=1$, $\rho(F_o)=1$, and $\B(F_o)=\B$.
Now the assertion folows by Proposition \ref{prop-non-Gor}.
\end{proof}


\end{document}